\documentclass[11pt]{article}
\usepackage{latexsym,amsmath,amsthm,verbatim,ifthen,amssymb,caption,float}
\usepackage[latin1]{inputenc}
\usepackage[all]{xy}
\usepackage{graphicx}
\usepackage{color}

\newdir{ >}{{}*!/-12pt/@{>}}
\newdir{ (}{!/-3pt/@_{(} } 

\newdir{ )}{!/-3pt/@^{(} } 
   
    \newcommand{\esecat}{{\rm e\text{-}secat}}
     \newcommand{\psecat}{{\rm p\text{-}secat}}
       \newcommand{\im}{\operatorname{{\rm Im}}}
    \newcommand{\secat}{{\rm secat}}
    \newcommand{\pcat}{\operatorname{{\rm p\text{-}cat}}}
     \newcommand{\id}{\operatorname{{\rm id}}}
      \newcommand{\ptc}{\operatorname{{\rm p\text{-}TC}}}
       \newcommand{\etc}{\operatorname{{\rm e\text{-}TC}}}
       \newcommand{\tc}{\operatorname{{\rm TC}}}
      \newcommand{\ecat}{\operatorname{{\rm e\text{-}cat}}}
      \newcommand{\pro}{\operatorname{\bf{P}}}
      \newcommand{\exte}{\operatorname{\bf{E}}}
      \newcommand{\topinf}{\operatorname{{\bf{Top}}^\infty}}

\newcommand{\cat}{\operatorname{{\rm cat}}}

\newcommand{\cale}{{\mathcal E}}

\newcommand{\ex}{\mathcal{E}} 
\newcommand{\T}{\tau}  

\newcommand{\R}{\mathbb{R}}

\newcommand{\nil}{\operatorname{\text{\rm nil}}}

\newtheorem{theorem}{Theorem}[section]

\newtheorem{corollary}[theorem]{Corollary}
 \newtheorem{lemma}[theorem]{Lemma}
 \newtheorem{proposition}[theorem]{Proposition}
 \theoremstyle{definition}
 \newtheorem{definition}[theorem]{Definition}
 \theoremstyle{definition}
 \newtheorem{remark}[theorem]{Remark}
 
 \numberwithin{equation}{section}

\newcommand{\br}{\mathbb R}

\newcommand{\sm}{{\setminus}}

\newcommand{\PP}{\textbf{P}}
\newcommand{\E}{\textbf{E}}

\newcommand{\extco}{H^*_{\cale}}

\begin{document}

\title{Proper topological complexity}

\author{Jos\'e Manuel Garc{\'\i}a-Calcines and Aniceto Murillo\footnote{The authors have been partially supported by the grant PID2023-149804NB-I00 of the Spanish Government and the Junta de Andaluc{\'\i}a grant ProyExcel-00827.}}

\maketitle

\begin{abstract}
We introduce and study the proper topological complexity of a given configuration space, a version of the classical invariant for which we require that the algorithm controlling the motion is able to avoid any possible choice of ``unsafe'' area. To make it a homotopy functorial invariant we characterize  it as a particular instance of the exterior sectional category of an exterior map, an invariant of the exterior homotopy category which is also deeply analyzed.
\end{abstract}


\section*{Introduction}
In broad terms, the topological complexity $\tc(X)$ of a motion  on a given configuration space $X$, as originally introduced in \cite{far}, provides an  lower bound of the number of ``continuous instructions'' of any algorithm controlling the given motion. One can also regard this invariant as the minimum number of navigational instabilities of the given motion.

However, to improve the efficiency of a given motion planner, one often need to impose extra conditions. Accordingly, different versions of topological complexity, sharpening the original one,  have appeared in the last years. For instance, for the {\em symmetric topological complexity} \cite{fargrant}, the motion from a point of the configuration space to itself is required to be the constant path and the trajectories given by the motion between two given points are the same regardless of which point is considered initial or final. Interesting  generalizations of this concept are described by different version of {\em equivariant topological complexity} \cite{colgrant,dra,lumar}. Another illustrative example is given by the {\em  efficient topological complexity} \cite{blaszca} where the path given by the motion between two given points of the configuration space is required to have the shortest possible length.

Here we impose that the algorithm controlling the motion  avoids any possible choice of ``unsafe environmental conditions'' of the given configuration space. As paradigmatic example, let $X$ be the configuration space on which we want to develop a motion planner and let $f\colon X\to Y$ be a continuous map describing certain environmental parameters of $X$. For instance, if $X$ is the configuration space of a navigational robot we could take $Y$ to be $\br^n$ and $f$ supplying the atmospheric pressure, humidity, temperature,..at each point of $X$.  Another case is providing by choosing $X$ to be the configuration space of an articulated arm and $f$ to be the associated work map.

Consider then $f^{-1}({\mathcal A})$, with ${\mathcal A}\subseteq Y$ an ``unsafe area'', given by the region of the configuration spaces in which $f$ takes values in  $\mathcal A$, thought of as a set of extremal conditions to be avoided.
For  safety purposes it is reasonable to ask our planner that for all $x,y\notin f^{-1}({\mathcal A})$ the whole path $\sigma(x,y)$ describing the trajectory from $x$ to $y$  is not exposed to these extremal conditions. Moreover, for practical purposes, we may assume that the unsafe area is compact and,  for efficiency purposes this time, the algorithm controlling the motion should operate independently of the chosen unsafe area.

Thus, the topological complexity of a configuration space measuring the instabilities of a planner which is ready to avoid any given unsafe area is formally translated to:

\begin{definition}\label{defmain}
 Let $X$ be a topological space. The {\em proper topological complexity} of $X$, denoted by $\ptc(X)$, is the minimum integer $n$ (or $\infty$ if no such integer exists) such that $X\times X$ is covered by $n+1$ open sets $\{U_0,\dots,U_n\}$ for which there are continuous local sections $\sigma_i\colon U_i\to X^I$ of the path fibration $q\colon X^I\to X\times X$ satisfying the following additional condition:

Every compact subspace $K$ of $X$ is contained in another compact $L$  such that $ \im\sigma_i(x,y)\subseteq K^c$ for all $x,y\in L^c$.
\end{definition}

The proper topological complexity is closely related to the following general problem in robotics: assume the configuration space of a given motion planning is a finite cell complex $Y$, and one requires the motion between two points in a certain subcomplex $Z$ of $Y$ to remain entirely in $Z$. In other words, $Z$ and $Y{\setminus} Z$ are considered the uniquely defined safe and unsafe regions, respectively. In this context, provided that $Y{\setminus}Z\hookrightarrow Y$ is a homotopy equivalence, the topological complexity of $Y$, subject to the mentioned restriction, is precisely $\ptc(Y{\setminus}Z)$ (see \S\ref{finale}).

\smallskip

The proper topological complexity, which clearly coincides with $\tc(X)$ whenever $X$ is compact, is, remarkably, an invariant of the proper homotopy type of $X$, despite the fact that the path fibration $q$ is not a proper map. To establish this and to extract the main properties of $\ptc$, we characterize it as a homotopy functorial invariant within the exterior homotopy category (see \S1 for a brief overview of exterior homotopy theory).  Namely,  and taking into account that the proper category is fully and faithfully embedded in the exterior category, we  set:

 \begin{definition}\label{secatmain}
 Given an exterior map $f\colon X\to Y$, the {\em exterior sectional category of $f$}, $\esecat(f)$ is the least integer $n$ (or $\infty$) for which there exists an exterior open covering $\{U_i\}_{i=0}^n$ of $Y$ and diagrams
 $$\xymatrix{
{U_i} \ar@{^{(}->}@<-2pt>[rr]\ar[dr]_{r_i}& & {Y} \\
 & {X} \ar[ur]_f }
 $$
 commuting up to exterior homotopy. As the paths $X^I$ of an exterior space $X$ is naturally endowed with an externology for which the path fibration $q\colon X^I\to X\times X$ is an exterior map, define the {\em exterior topological complexity of $X$} as
 $$
\etc(X)=\esecat(q).
$$
\end{definition}
The exterior sectional category and its proper counterpart are deeply analized in \S2. We find specially interesting the subaditivity of this invariant and its consequences when particularizing to the proper category.

On the other hand, the exterior topological complexity, together with its proper version, which constitutes the core of this work, is studied in \S3. In particular, we prove (see Proposition \ref{princi}):

\begin{theorem}
$$
 \ptc(X)=\etc(X_{cc})=\esecat(q\colon X_{cc}^I\to X_{cc}\times X_{cc})
 $$
 where $X_{cc}$ denotes the space $X$ endowed with the cocompact externology. As in the exterior homotopy category $q$ can be replaced by the diagonal we also have
 $$
 \ptc(X)=\esecat(\Delta\colon X_{cc}\to X_{cc}\times X_{cc}).
 $$
\end{theorem}

At this point we can  obtain the main properties and bounds of this proper invariant summarized in Theorems \ref{corotcp} and  \ref{noncompact}, as well as an illustrative collection of examples in \S4.

For all of the above we have considered useful to briefly include in \S1 the basics of exterior homotopy theory and its numerical invariants of Lusternik-Schnirelmann type we will need.

\section{Preliminaries}

The following constitutes standard facts on  exterior homotopy theory for which we refer to \cite{calpiher} or the summaries   in \cite[\S2]{gargarmu2} or \cite[\S1]{gargarmu4}.

\medskip

 Denote by
$\mathbf{P}$ the category of spaces and proper maps (which are always assumed to be continuous) and write $\mathbb{R}_+=[0,\infty)$. A {\em ray} of $X$ is a proper map $\br_+\to X$, which plays in $\PP$ the role of choosing a base point. Let $\PP_\infty\subset \PP$ denote the full subcategory of $\PP$ consisting of Hausdorff, non compact, $\sigma$-compact and locally compact spaces. For any $X\in\PP_\infty$  there is a proper map $r\colon X\to \br_+$ which is unique up
to proper homotopy. A {\em strong end} of a given $X\in\PP$ is a proper homotopy class of a ray.

\medskip

 An \textit{exterior space}
$(X,\ex)$ consists of a topological space $(X,\T)$
together with a non empty family of {\em exterior sets} $\ex\subseteq \T $, called
\textit{externology}, which  is closed by finite intersections and,
whenever $U \supseteq E$, $E \in \ex$, $U\in \T$, then $U\in \ex$.  Note that, since $\ex$ is non empty,  $X$ is always an exterior set. The {\em total externology} is  $\ex= \T $ and the {\em trivial externology} is $\ex =\{X\}$. In what follows we denote by $X_{tot}$ and $X_{tr}$ the space $X$ endowed, respectively, with the total and trivial externology. Note that a given externology is the total one if and only if the empty set is exterior.

 An externology on a space can be thought of as  a  ``neighborhood basis at infinity''. A continuous map $f\colon (X,\ex ) \rightarrow (X',\ex'
)$ is  \textit{exterior} if  $f^{-1}(E) \in \ex$, for all $E \in \ex'$. The {\em relative externology} in a subspace $Y\subseteq X$ of an exterior space is defined as the coarsest externology on $Y$ that makes the inclusion map exterior. We denote by $\E$ the category
 of exterior spaces and maps which is complete and cocomplete \cite[Theorem 3.3]{calpiher}.

Any topological space $X$  can be endowed with the
\textit{cocompact externology}  formed by the
family of the complements of all closed compact subspaces. The
corresponding exterior space will be denoted by $X_{cc}$. This defines a full embedding
\begin{equation}\label{fullembed} \PP \hookrightarrow \E,\qquad X\mapsto X_{cc}.\end{equation}

The product $X\times Y$ of two exterior spaces is always considered with the product externology given by the open sets of the product which contain a product  of exterior open sets of $X$ and $Y$.

In particular, for any exterior space $X$ consider the exterior product $X\times I_{tr}$ where $I=[0,1]$. This is the {\em exterior cylinder} of $X$ and provides the right notion of exterior homotopy  between two exterior maps which will be denoted by $\simeq_e$. Whenever a given space $X$ is endowed with the cocompact externology one has
\begin{equation}\label{cilindro}
X_{cc}\times I_{tr}=(X\times I)_{cc}
\end{equation}
 and thus it produces the genuine cylinder in $\PP$. Note also that, for any space $X$,
 \begin{equation}\label{totaltrivial}
 X_{tot}\times I_{tr}=(X\times I)_{tot}
 \end{equation}

On the other hand, for any exterior space $X$, the path space $X^I$ is endowed with the externology formed by the open sets containing some $(I,E):=\{\alpha \in X^I,\,\alpha (I)\subseteq E\}$ with $E\subseteq X$ exterior.

A ray of an exterior space $X$ is an exterior map ${\br_+}\to X$ in which $\br_+$ is considered with the cocompact externology. We denote by $\exte_0$ the category  of exterior spaces $X$ admitting a ray and for which there is an exterior map $r\colon X\to\br_+$ which is then necessarily unique up to exterior homotopy.

Note that, for any $X\in\exte_0$, the projection
$p\colon X\times\br_+\stackrel{\simeq_e}{\longrightarrow} X$ is an exterior homotopy equivalence.  Indeed, consider the exterior map
$(\id_X,r)\colon X\rightarrow X\times \mathbb{R}_+$ where $r\colon X\to\br_+$ is a given exterior map. Then $p\circ
(\id_X,r)=\id_X$ and  $(\id_X,r)\circ p\simeq_e\id_{X\times \mathbb{R}_+}$  through the usual homotopy which happens to be also exterior
$F\colon (X\times \mathbb{R}_+)\times I_{tr}\rightarrow X\times
\mathbb{R}_+$, $F\bigl((x,t),s\bigr)=(x,(1-s)r(x)+st)$.

Finally, for readers who are less familiar with exterior spaces, we outline the following, which may be helpful for understanding the arguments in the subsequent sections.  Let $\topinf$ be the subcategory of the category of pointed spaces $(X,x_0)$ in which the base point is a closed set and maps $f\colon (X,x_0)\to (Y,y_0)$ for which $f^{-1}(\{y_0\})=\{x_0\}$. Then, see \cite[Proposition 6]{cal}, there is an equivalence of categories
$$
(-)^\infty \colon \exte\stackrel{\cong}{\longrightarrow}\topinf
$$
which assigns to each exterior space $(X,\tau,\ex)$ the pointed space $X^\infty=(X\sqcup\{\infty\},\infty)$ equipped with the topology $\tau\cup\{E\sqcup\{\infty\},\,E\in\ex\}$. For any exterior map $f$, $f^\infty$ is obviously defined. Being an equivalence, it preserves small limits and colimits. In particular
$$
(X\times Y)^\infty=X^\infty\wedge Y^\infty
$$
as the smash is the categorical product in $\topinf$.
Moreover, it can be shown that, for any exterior space $X$, its exterior cylinder $X\times I_{tr}$ is sent (up to isomorphism in $\topinf$) to $X^\infty\times I/\{\infty\}\times I$ through this equivalence.

 Note also that the above equivalence fits in the commutative diagram
 $$
\xymatrix{\PP\,
\ar@{^{(}->}[r] \ar[dr]_{(-)^+}
                &    \exte\ar[d]^{(-)^\infty}    \\
                & \topinf  }
$$
where the horizontal arrow is the  embedding in (\ref{fullembed}) and $(-)^+$ is the obvious functor induced by the Alexandroff compactification, which is therefore a full embedding.

\medskip

 We now briefly recall the Lusternik-Schnirelmann category of based objects in the proper and exterior setting.

Given $X\in\PP_\infty$, the \textit{proper Lusternik-Schnirelmann category of $X$}, denoted by
$\pcat{X}$, is the smallest number $n$ (or infinite) for which $X$ admits an open  cover $\{U_0, ... , U_n\}$
such that, for each $i=0,\dots,n$, there is a diagram in $\PP_\infty$
$$\xymatrix@R=0.6cm@C=0.5cm{
  \overline{U}_i \,\ar@{^{(}->}@<-2pt>[rr]  \ar[dr]_{r} &  &    X    \\
               & \R_+   \ar[ur]_{\alpha_i}     }$$
                which commutes up to proper homotopy \cite{adomarquin}. Here, $\overline{U}_i$ denotes the closure of $U_i$.

                On the other hand, see \cite[Definition 3.2]{carlasmuquin}, given $\alpha$  a ray of $X$,
 the \textit{$\alpha$-proper Lusternik-Schnirelmann category of $X$}, denoted by
$\pcat_{\alpha}{X}$, is the smallest number $n$ (or infinite) for which $X$ admits an open  cover $\{U_0, ... , U_n\}$
such that, in the above diagram $\alpha_i$ may be chosen to be $\alpha$ for all $i=0,\dots,n$.
 This number does not depend on the proper homotopy class of $\alpha$ and the inequality
\begin{equation}\label{inecutri}
 \pcat(X)\le \pcat_{\alpha}(X)
\end{equation}
 is obvious. For instance  $\br$   has two ends and $\pcat_\alpha(\br)=\infty$ for any ray $\alpha$, while $\pcat(\br)=1$. Contrary, for $n\ge2$, $\br^n$ is strongly one ended, i.e., there is just one ray in X, up to proper homotopy, and $\pcat_\alpha(\br^n)=\pcat(\br^n)=1$ for any ray $\alpha$ \cite{adomarquin}.

In the exterior setting,
given $X\in \exte_0$ and a ray $\alpha$ of $X$, the \textit{$\alpha$-exterior Lusternik Schnirelmann category} of $X$, denoted by
$\ecat_{\alpha}{X}$, is the smallest number $n$ (or infinite) for which $X$ admits an open  cover $\{U_0, ... , U_n\}$
such that, for each $i=0,\dots,n$, there  is a diagram
\begin{equation}\label{uff}
\xymatrix{
 U_i \ar@{^{(}->}[rr] \ar[dr]_{r}
                &  &    X    \\
                & {\mathbb{R}_+} \ar[ur]_{\alpha }   }
                \end{equation}
commuting up to exterior homotopy. Note that, since each $U_i$ is considered with the induced externology (and not the cocompact one) there is no need to consider its closure as is done in defining $\pcat$.  For any $X\in\PP_\infty$, and any ray $\alpha$ of $X$,
\begin{equation}\label{catego}
\pcat_{\alpha}({X})=\ecat_{\alpha}({X_{cc}}).
\end{equation}

\section{Exterior and proper sectional category}

  In order to give a homotopy functorial description of the proper  topological complexity we  need to regard it  as a particular instance of the sectional category in the proper, or more generally, exterior context, which we now introduce and study. For the arguments to remain consistent, we require all spaces to be Hausdorff and completely normal. This means that any two separated subsets must be separable by neighborhoods or, equivalently, that every subspace is normal. Additionally, we need the product of spaces to inherit these properties. Therefore, we restrict our discussion to the class of metrizable spaces, which naturally satisfy these conditions.

\begin{definition}
Let $f\colon X\rightarrow Y$ be an exterior map. An open subset
$U\subseteq Y$, non necessarily exterior, is said to be \emph{e-sectional} if there
exists, up to exterior homotopy, a commutative diagram in
$\mathbf{E}$ of the form
$$\xymatrix{
{U} \ar@{^{(}->}@<-2pt>[rr]\ar[dr]& & {Y} \\
 & {X} \ar[ur]_f }$$
The \emph{exterior sectional category} of $f$, denoted by
$\esecat(f)$, is the least integer $k$ such that $Y$ can
be covered by $k+1$  e-sectional open subsets. If no such $k$
exists, then we set $\esecat(f)=\infty .$
\end{definition}

\begin{remark}\label{rem1} Let $\alpha\colon \br_+\to X$ be a ray of the exterior space $X$. Then, by definition,
$$
\ecat_{\alpha}(X)=\esecat(\alpha).
$$
\end{remark}

In what follows, $\secat$ denotes the classical sectional category.

\begin{proposition}\label{propo1} For any exterior map $f\colon X\to Y$, $\secat(f)\le \esecat(f)$. Equality holds if $Y$ is endowed with the total externology.
\end{proposition}

\begin{proof} The first statement is elementary. For the second note first that $X$ has also the total externology as the empty set, being $f^{-1}(\emptyset)$, is exterior. Let $s\colon U\to X$ be a homotopy local section
of $f$ and  $H\colon U\times
I\rightarrow Y$ a homotopy between the inclusion $U\hookrightarrow Y$ and $f\circ s$. In view of (\ref{totaltrivial}), $Y_{tot}\times I_{tr}=(Y\times I)_{tot}$ and thus $U_{tot}\times I_{tr}=(U\times I)_{tot}$ so that $H$ is an exterior homotopy and $U$ is e-sectional. This shows that $\esecat(f)\le \secat(f)$ and equality holds.
\end{proof}

\begin{remark} Strict
 inequality may occur in the previous result. For instance, if we endow $\br^n$ with the cocompact externology,  $i\colon \mathbb{R}_+\hookrightarrow \mathbb{R}^n$ is not an exterior homotopy equivalence. Moreover, by Remark \ref{rem1}, $\esecat(i)=\pcat_i(\br^n)$ which is $1$ if $n\ge2$ and $\infty$ if $n=1$. However $\secat(i)=0$.
\end{remark}

\begin{lemma}\label{einv-1}
Given
$$\xymatrix{
{X} \ar[rr]^{\lambda } \ar[dr]_f & & {X'} \ar[dl]^{f'} \\
 & {Y} & }$$
an exterior homotopy commutative diagram,
 $\esecat(f')\leq \esecat(f)$. In
particular, if $\lambda $ is an exterior homotopy equivalence,
$\esecat(f')=\esecat(f).$
\end{lemma}

\begin{proof}
Simply note that, if  $U\subseteq Y$ is an e-sectional open set with respect to $f$ through the exterior homotopy local section
$s\colon U\rightarrow X$, then
$\lambda \circ s\colon U\rightarrow X'$ is an exterior homotopy local section of $f'$.
\end{proof}

Combining this lemma with Remark \ref{rem1} we obtain:

\begin{corollary}\label{coro1}
Let $f\colon X\rightarrow Y$ be an exterior map and $\alpha\colon\br_+\to X$ a ray of $X$. Then,
$$\esecat(f)\leq \ecat_{f\circ \alpha }(Y)$$
\hfill$\square$
\end{corollary}

\begin{lemma}\label{einv-2}
Let
$$\xymatrix{
 & {X} \ar[dl]_f \ar[dr]^{f'} \\
{Y} \ar[rr]_{\mu }^{\simeq_{e} } & & {Y'} & }$$
be an exterior homotopy commutative diagram where $\mu $
is an exterior homotopy equivalence. Then,
$\esecat(f)=\esecat(f').$
\end{lemma}

\begin{proof}
Simply observe that if $U\subseteq Y'$ is an e-sectional open set with respect to $f'$ then $\mu^{-1}(U)$ is e-sectional with respect to $f$.
\end{proof}

An immediate consequence is the following:

\begin{proposition}\label{propo2}
Let
$$\xymatrix{
{X} \ar[rr]^{\lambda }_{\simeq_{e} } \ar[d]_f  & & {X'} \ar[d]^{f'} \\
{Y} \ar[rr]_{\mu }^{\simeq_{e}} & & {Y'} }$$ an exterior homotopy commutative diagram where $\lambda
$ and $\mu $ are exterior homotopy equivalences. Then
$\esecat(f)=\esecat(f').$
\end{proposition}

In the exterior setting we also have the usual upper bound for Lusternik-Schnirelmann type invariants of products. The proof, which we summarize here for completness, follows the same argument that the one in \cite[Theorem 1.37]{C-L-O-T}.

\begin{proposition}\label{producto}
For any pair $f,g$ of exterior maps,
$$
\esecat(f\times g)\le\esecat(f)+\esecat(g).
$$
\end{proposition}

\begin{proof} We remark first that for any exterior map $h\colon X\to Y$ is easy to see that $\esecat(h)\leq k$ if and only if there exists a nested sequence of open sets $\emptyset =V_0 \subseteq V_1\subseteq \dots\subseteq V_{k+1}=Y$ such that $V_i{\setminus} V_{i-1}$ is contained in an e-sectional open set of  $h$.

Write $f\colon E\to B$, $g\colon E'\to B'$, $\mbox{e-secat}(f)=n$ and $\mbox{e-secat}(g)=m$. We then may construct sequences of open sets
$$\emptyset =O_0\subseteq O_1\subseteq \cdots \subseteq O_{n+1}=B$$
$$\emptyset =P_0\subseteq P_1\subseteq \cdots \subseteq P_{m+1}=B'$$
\noindent so that $O_{i+1}{\setminus} O_i\subseteq U_{i+1},$ $P_{j+1}{\setminus} P_j\subseteq W_{j+1}$, where $U_{i+1}$ and $W_{j+1}$ are, respectively,  e-sectional open sets of $f$ and $g$.

For $r\geq 1,$ define
$$Q_r=\bigcup _{j=1}^rO_j\times P_{r+1-j}\subseteq B\times B'$$
with $O_j=\emptyset $ if $j>n+1$ and $P_t=\emptyset $ if $t>m+1$. If we set $Q_0=\emptyset $, the increasing sequence $Q_0\subseteq Q_1\subseteq \cdots \subseteq Q_{n+m+1}$ satisfies
$$Q_{j+1}{\setminus} Q_j=\bigcup _{k=1}^{j+1}(O_k{\setminus} O_{k-1})\times (P_{j+2-k}{\setminus} P_{j+1-k}).$$
Moreover, $(O_k{\setminus} O_{k-1})\times (P_{j+2-k}{\setminus} P_{j+1-k})\subseteq U_k\times W_{j+2-k},$ the latter being open and e-sectional of $f\times g$. As $j+2-k$ decreases as $k$ increases, it follows that  $(O_k{\setminus} O_{k-1})\times (P_{j+2-k}{\setminus} P_{j+1-k})$ and $(O_l{\setminus} O_{l-1})\times (P_{j+2-l}{\setminus} P_{j+1-l})$, being separated sets, can be separated by  disjoint open neighborhoods (recall that all considered spaces are assumed to be metrizable and therefore completely normal). We intersect these neighborhoods with $U_k\times W_{j+2-k}$ and $U_l\times W_{j+2-l}$ to obtain disjoint open e-sectional neighborhoods of $(O_k{\setminus} O_{k-1})\times (P_{j+2-k}{\setminus} P_{j+1-k})$ and $(O_l{\setminus} O_{l-1})\times (P_{j+2-l}{\setminus} P_{j+1-l})$ for $k\neq l$. Since these are e-sectional and disjoint, their union is also e-sectional. Hence, $$(O_k{\setminus} O_{k-1})\times (P_{j+2-k}{\setminus} P_{j+1-k})\bigcup (O_l{\setminus} O_{l-1})\times (P_{j+2-l}{\setminus} P_{j+1-l})$$ \noindent is contained in an open e-sectional subset. We iterate this argument  to show that $Q_{j+1}{\setminus} Q_j$ is contained in an open e-sectional subset. From there one easily build an e-sectional covering of $f\times g$ of $n+m+1$ elements.
\end{proof}

\begin{corollary}\label{catepro}
Given $X,Y\in \exte_0$ with rays  $\alpha$ and $\beta$ respectively,
$$
\ecat_{(\alpha,\beta)}(X\times Y)\le \ecat_\alpha(X)+\ecat_\beta(Y).
$$
\end{corollary}

\begin{proof} By definition $\ecat_{(\alpha,\beta)}(X\times Y)=\esecat(\alpha,\beta)$ which coincides with $\esecat(\alpha\times\beta)$ by Lemma \ref{einv-1} in view of the diagram
$$\xymatrix{
{\br_+} \ar[rr]^(.41){\Delta }_(.43){\simeq_e} \ar[dr]_{(\alpha,\beta)} & & {\br_+\times\br_+.} \ar[dl]^{\alpha\times\beta} \\
 & {X\times Y} & }$$
To finish, apply Proposition \ref{producto}.
 \end{proof}

\begin{remark}\label{remi}
The previous corollary does not let us conclude that\break $
\pcat_{(\alpha,\beta)}(X\times Y)\le \pcat_\alpha(X)+\pcat_\beta(Y)$ for any $X,Y\in \PP_\infty$ with rays  $\alpha$ and $\beta$ respectively. By formula (\ref{catego}) we do get that $\pcat_{(\alpha,\beta)}(X\times Y)=\ecat _{(\alpha,\beta)}(X\times Y)_{cc}$. However, this cannot be compare to   $\ecat_{(\alpha,\beta)}(X_{cc}\times Y_{cc})$ as $(X\times Y)_{cc}$ and $X_{cc}\times Y_{cc}$ are not, in general, of the same exterior homotopy type. In fact, the identity
$$
X_{cc}\times Y_{cc}\stackrel{\id}{\longrightarrow} (X\times Y)_{cc}
$$
is an exterior map but, in general, it fails to be exterior in the other direction.

For the first assertion, let
  $E$ be an exterior open set of $ (X\times Y)_{cc}$ so that $(X\times X){\setminus} E$ is compact, so are its projections  $K_X\subseteq X$ and $K_Y\subseteq Y$ over $X$ and $Y$ respectively. Then, $E\supseteq E_X\times E_Y$ where $E_X=X{\setminus} K_X$ and $E_Y=Y{\setminus} K_Y$.

\end{remark}
 Nevertheless, a slight reformulation of \cite[Proposition 2.1]{carlasquin} proves:

 \begin{proposition}\label{cateprop}
Given connected polyhedra $X,Y\in \PP_\infty$ with rays  $\alpha$ and $\beta$ respectively,
$$
\pcat_{(\alpha,\beta)}(X\times Y)\le \pcat_\alpha(X)+\pcat_\beta(Y).
$$
\end{proposition}
Recall that a polyhedron is a space homeomorphic to the geometric realization of a simplicial complex. A polyhedron in $\PP_\infty$, being non compact but locally compact, is necessarily the realization of a non finite but locally finite simplicial complex.
\begin{proof}
By Proposition 2.1 of op. cit.
$$
\pcat(X\times Y)\le \pcat(X)+\pcat(Y).
$$
Moreover, as $X\times Y$ is strongly one-ended \cite[Theorem 2.2]{mi}, $\pcat(X\times Y)=\pcat_{(\alpha,\beta)}(X\times Y)$. The result follows from the trivial inequality (\ref{inecutri}).
\end{proof}

The following extends the previous result whenever one of the spaces is compact and it a consequence of previous results, including Proposition \ref{producto}. Let $X\in \PP_\infty$ with ray  $\alpha$ and let $Y$ be a path-connected compact space. Consider the ray $\alpha'$ in $X\times Y$ given by $\alpha'(t)=(\alpha(t),y_0)$ for a fixed point $y_0\in Y$. Then:

\begin{proposition} $\pcat_{\alpha'}(X\times Y)\le\pcat_{\alpha}(X)+\cat(Y)$.
\end{proposition}
\begin{proof} Since $Y$ is compact it follows that $(X\times Y)_{cc}=X_{cc}\times Y_{tr}$, where $Y_{tr}$ denotes the space $Y$ endowed with the trivial externology ${\cal E}=\{Y\}$. On the other hand, see  \S1, the map $X_{cc}\to X_{cc}\times \br_+$, $x\mapsto\bigl(x,r(x)\bigr)$ is an exterior homotopy equivalence, where  $r\colon X\to \br^+$ is the unique proper map, up to proper homotopy. Hence, the map
$$
\omega\colon (X\times Y)_{cc}\longrightarrow X_{cc}\times(\br_+\times Y)_{cc},\quad (x,y)\mapsto(x,r(x),y),
$$
describes the exterior homotopy equivalence,
$$
(X\times Y)_{cc}\simeq_e(X_{cc}\times\br^+)\times Y_{tr}=X_{cc}\times(\br_+\times Y_{tr})=X_{cc}\times(\br_+\times Y)_{cc},
$$
and it fits in the commutative diagram of exterior spaces and maps
$$\xymatrix{
 & {\mathbb{R}_+} \ar[dl]_{\alpha '} \ar[dr]^{(\alpha,\beta)} &  \\
 {(X\times Y)_{cc}} \ar[rr]^{\simeq _e}_{\omega } & & {X_{cc}\times (\mathbb{R}_+\times Y)_{cc}}  }$$
 where $\beta(t)=(t,y_0)$.
 Finally, in view of the equality (\ref{catego}), Remark \ref{rem1}, Lemma \ref{einv-2} and Proposition \ref{producto}, we have:
$$\begin{array}{ll}
\pcat_{\alpha '}(X\times Y) & =\ecat_{\alpha '}(X\times Y)_{cc} \\
                                   & =\ecat_{(\alpha,\beta)}(X_{cc}\times (\mathbb{R}_+\times Y)_{cc}) \\
                                   & \leq \ecat_{\alpha}(X_{cc})+\ecat_{\beta}(\mathbb{R}_+\times Y)_{cc} \\
                                   & =\pcat_{\alpha }(X)+\pcat_{\beta}(\mathbb{R}_+\times Y) \\
                                   & =\mbox{p-cat}_{\alpha }(X)+\mbox{cat}(Y).

\end{array}$$
\end{proof}

We now turn to the classical cohomology lower bound of the sectional category in the exterior context. For it,  we choose any cohomology theory  of ``Eilenberg-Steenrod kind'' or {\em ES-theory} (see for instance \cite[\S3]{spa}) with coefficients in an arbitrary ring $R$. This is a contravariant functor $H^*$ from the category of pairs of spaces to
the category of non-negatively graded $R$-modules satisfying exactness, excision, homotopy
invariance and continuity. Then, see \cite[\S3]{gargarmu4}, define the {\em exterior cohomology} of an exterior pair $(X,A)$ as,
$$
H^*_{\cale}(X,A)=\varinjlim \{H^*(X,E),\,\, \text{$E$ exterior neighborhood of $A$}\}.
$$
By \cite[Theorem, 3.2]{spa} and \cite[Remark 3.2(i)]{gargarmu4}, this defines a ES-theory in the category of pairs of exterior spaces. As usual, given pairs of exterior spaces  $(X,A)$ and $(X,B)$, the exterior cohomology of the diagonal $(X,A\cup B)\to (X\times X,A\times X\cup X\times B)$ composed with the cross product $H^*_{\cale}(X,A)\times \extco(X,B)\to\extco(X\times X,A\times X\cup X\times B)$
produces the cup product
$
\cup \colon H^*_{\cale}(X,A)\otimes \extco(X,B)\to\extco(X,A\cup B)$.
Then, we have the following result whose proof is completely analogous to the corresponding statement in the classical setting \cite[Proposition 9.14(3)]{C-L-O-T}.

\begin{proposition}\label{cohomo} Let $f\colon X\to Y$ be an exterior map. Then,
$$
\nil\ker\extco(f)\le \esecat(f).
$$
In other words, given exterior cohomology classes $\gamma_1,\dots,\gamma_k\in \extco(Y)$ such that $\extco(f)(\gamma_i)=0$ for all $i$ and $\gamma_1\cup\dots\cup \gamma_k\not=0$, then $\esecat(f)\ge k$.
\hfill$\square$
\end{proposition}

\medskip

At this point we  could introduce the proper sectional category of a map either by declaring it to be its exterior sectional category with respect to the cocompact externologies, or  following the classical approach to other invariants  of Lusternik-Schnirelmann type in the proper homotopy category \cite{adomarquin}. Namely:

\begin{definition}
Let $f\colon X\rightarrow Y$ be a proper map. A closed subset
$C\subseteq Y$ is said to be \emph{properly sectional} if there
exists a proper homotopy commutative diagram in
$\mathbf{P}$ of the form
$$\xymatrix{
{C} \ar@{^{(}->}@<-2pt>[rr]\ar[dr]_s & & {Y} \\
 & {X} \ar[ur]_f }$$
The \emph{proper sectional category} of $f$, denoted as
$\mbox{p-secat}(f),$ is the least integer $k$ such that $Y$ can
be covered by $k+1$ open subsets $\{U_0,\dots,U_k\}$ such that the closure $\overline{U_i}$ is properly sectional
for any $i=0,1,\dots,k$. If no such $k$ exists,
 we set $\mbox{p-secat}(f)=\infty .$
\end{definition}

\begin{theorem}\label{bridge}
For any proper map $f\colon X\rightarrow Y$,
$$\psecat(f)=\esecat(f_{cc}).
$$
\end{theorem}

\begin{proof}
Let $U\subseteq Y$ be an open subset together
with a proper map $s\colon \overline{U}\rightarrow X$ making
commutative, up to proper homotopy, the triangle
$$\xymatrix{
{\overline{U}} \ar@{^{(}->}@<-2pt>[rr]^i\ar[dr]_s & & {Y.} \\
 & {X} \ar[ur]_f }$$
Let $H\colon\overline{U}\times I\rightarrow
Y$ be a proper homotopy between $i$ and $f\circ s$ and consider the
exterior map $H_{cc}\colon (\overline{U}\times
I)_{cc}\rightarrow Y_{cc}$. We recall from (\ref{cilindro}) that $(\overline{U}\times
I)_{cc}=\overline{U}_{cc} \times I_{tr}$ and thus, with the    externology induced
by $(\overline{U}\times
I)_{cc}$, $U\times I$ is precisely the
exterior space $U \times  I_{tr}$ (here we consider on $U$ the
induced externology with respect to $X_{cc}$). Therefore, the
restriction of $H_{cc}$ to $U\times I=U\times I_{tr}$ makes $U$ e-sectional with respect to $f_{cc}$. This shows that $\esecat(f_{cc})\leq
\psecat(f)$.

\medskip Conversely, let $s\colon U\to X$ be a homotopy exterior local section of $f_{cc}$ over the open set $U\subseteq Y$, let
 $H\colon U\bar{\times }I\rightarrow
Y_{cc}$ be an exterior homotopy between $U\hookrightarrow Y$ and  $f_{cc}\circ s$, and let
 $V\subseteq Y$ be an open set such that
$\overline{V}\subseteq U$.  Taking into account again that the
 externology in $\overline{V}\times I$  induced by
$U\bar{\times} I$ gives the exterior space
$(\overline{V})_{cc}\times I_{tr}=(\overline{V}\times I)_{cc}$,
the restriction   $H\colon \overline{V}\times I\rightarrow Y$ is a proper map which makes $V$ properly
sectional. Finally,  since any open cover
$\{U_i\}_{i=0}^n$ of $Y$ admits a refinement
$\{V_i\}_{i=0}^n$ such that $
\overline{V}_i\subseteq U_i$ for all $i=0,1,...,n$ (recall that every considered space is assumed to be metrizable and thus  normal), it follows that
$\psecat(f)\leq \esecat(f_{cc})$.
\end{proof}

\begin{corollary}\label{coro2} Let $f\colon X\to Y$ be a proper map. Then:
\begin{itemize}
\item[(i)] $\psecat(f)$ is an invariant of the proper homotopy type of $f$.

\item[(ii)]  $\secat(f)\le \psecat(f)$ and equality  holds if $Y$ is compact.

\item[(iii)] For any ray $\alpha\colon \br_+\to X$ and provided $f\in\PP_\infty$,
$$
\psecat(f)\le\pcat_{f\circ\alpha}(Y)\quad\text{and}\quad \pcat_{\alpha}(X)=\psecat(\alpha).
$$
\item[(iv)] $\nil\ker H^*_{cc}(f)\le \psecat (f)$.
\end{itemize}
\end{corollary}
In {\em (iv)}, and following the notation in Proposition \ref{cohomo}, $H^*_{cc}$ is the exterior cohomology with respect to the cocompact externology induced by a given ES-theory $H^*$. In other words, it is the compact supported cohomology induced by $H^*$.
\begin{proof}
Items {\em (i)}, {\em (ii)} and {\em (iv)}  are simply  reformulations of Propositions  \ref{propo2}, \ref{propo1} and    \ref{cohomo}, respectively, by choosing the cocompact externology.

In the same way,   the first and second inequalities of {\em (iii)} follow, respectively, from Corollary \ref{coro1} and Remark \ref{rem1} taking into account the equality (\ref{catego}).
\end{proof}

\section{Exterior and proper topological complexity}

In this section we introduce the exterior topological complexity. More than a simple extension of the classical concept, it provides a way of regarding under a functorial point of view the proper topological complexity of Definition \ref{defmain} which we restate below.  Consequently, this leads to a  collection of properties  and examples of this proper invariant which would be much harder to obtain, if attainable,  staying within  the sole proper homotopy category.

\begin{definition}\label{defmain2}
 The {\em proper topological complexity} $\ptc(X)$ of  a space $X$ is the least integer $n$ (or $\infty$ if no such integer exists) such that $X\times X$ is covered by $n+1$ open sets $\{U_0,\dots,U_n\}$ for which there are continuous local sections $\sigma_i\colon U_i\to X^I$ of the path fibration $q\colon X^I\to X\times X$ satisfying:

 Every compact subspace $K$ of $X$ is contained in another compact $L$  such that $\im\sigma _i(x,y)\subseteq K^c$, for all $(x,y)\in U_i\cap (L^c\times L^c)$.
\end{definition}

On the other hand, in the exterior context, recall first that an exterior fibration is an exterior map satisfying the exterior homotopy lifting property \cite[Definition 2.5]{gargarmu1}. Then, we set:

\begin{definition}\label{ptc} Let $X$ be an exterior space. The {\em exterior topological complexity} of $X$, denoted by $\etc(X)$, is the minimum integer $n$ (or $\infty$ if no such integer exists) such that $X\times X$ is covered by $n+1$ open sets $\{U_0,\dots,U_n\}$ for which there are exterior local sections $\sigma_i\colon U_i\to X^I$ of $q\colon X^I\to X\times X$.
Note that, since $q$ is an exterior fibration,
$$
\etc(X)=\esecat(q).
$$
\end{definition}\label{igualdad}
As an essential instance  we have:
\begin{proposition}\label{princi}
$
\ptc(X)=\etc(X_{cc}).
$
\end{proposition}

\begin{proof}
It is enough to check that a continuous local section $\sigma\colon U\to X^I$ of $q$ is exterior (with respect to the cocompact externology) if and only if it satisfies that every compact subspace $K$ of $X$ is contained in another compact $L$  such that $ \im\sigma(x,y)\subseteq K^c$ for all $x,y\in L^c$.

Let $\sigma\colon U\to X_{cc}^I$ be an exterior local section of $q$ and  let $K\subseteq X$ be a compact subspace. Since $(I,X{\setminus} K)\subseteq X^I$ is exterior, $\sigma^{-1}\bigl((I,X{\setminus} K)\bigr)$ is exterior and hence, it contains a open set of the form $U\cap(X{\setminus} K_1)\times (X{\setminus} K_2)$ for some compacts subspaces $K_1,K_2\subseteq X$.

 Define
 $$
 L=K_1\cup K_2\quad\text{so that}\quad
L^c= K_1^c\cap K_2^c.
$$
Hence, $\sigma\bigl((L^c\times L^c)\cap U\bigr)\subseteq (I,K^c)$ which translates to $\im\sigma(x,y)\subseteq K^c$ for all $x,y\notin L$.

Conversely, given a continuous local section $\sigma \colon U\to X^I$ satisfying the additional condition, and given $K\subseteq X$ compact, this extra condition provides a compact $L$  containing $K$ and such that  $\sigma\bigl((L^c\times L^c)\cap U\bigr)\subseteq (I,K^c)$. That is, $\sigma \colon U\to X_{cc}^I$ is  exterior.
\end{proof}

\begin{remark}\label{nuevo} (1) For a given exterior space $X$, the exterior path fibration $q$ has the usual exterior factorization
$$\xymatrix{
X \ar[rr]^(.43){\Delta_X} \ar[dr]_(.43){\simeq_{e} } & &
X \times X \\
 & X^I \ar[ur]_{q}
}$$
since the diagonal $\Delta_X$ is exterior and the map $X\stackrel{\simeq_{e}}{\longrightarrow }X^I$ sending any point $x$ to the constant path on $x$ is an exterior homotopy equivalence. Hence, by Proposition \ref{propo2},
\begin{equation}\label{etcdiago}
\etc(X)=\esecat(\Delta_X).
\end{equation}
In particular, for any topological space $X$,
\begin{equation}\label{ptcdiago}
\ptc(X)=\esecat(\Delta_X\colon X_{cc}\to X_{cc}\times X_{cc}).
\end{equation}

 (2) However, observe that for a given topological space $X$, the path fibration $q$ is not proper in general so its ``proper sectional category'' is not even defined. On the other hand, the diagonal $\Delta_X\colon X\to X\times X$ is always proper which translates to $\Delta_X\colon X_{cc}\to (X\times X)_{cc}$ being exterior.

Nonetheless, $\psecat(\Delta_X)$, which  by Theorem \ref{bridge} is  $\esecat(\Delta_X\colon X_{cc}\to (X\times X)_{cc})$, is  different from $\esecat(\Delta_X\colon X_{cc}\to X_{cc}\times X_{cc}) $, which coincides with $\ptc(X)$ by (\ref{ptcdiago}), as $(X\times X)_{cc}$ has not the exterior homotopy type of $X_{cc}\times X_{cc}$, see Remark \ref{remi}.
 \end{remark}

In the following statements we enumerate the main and basic properties of the exterior and proper topological complexity.

\begin{theorem}\label{propie}
\begin{itemize}
\item[(i)] $\etc$ is an invariant of the exterior homotopy type.
\item[(ii)] For any exterior space $X$, $\tc(X)\le \etc(X)$ and equality holds if $X$ is endowed with the total externology.
    \item[(iii)] If $X$ has the exterior homotopy type of $\br_+$, then $\etc(X)=0$. The converse also holds for any $X\in\exte_0$.
        \item[(iv)] For any exterior space $X\in\exte_0$ and any  ray $\alpha$ of $X$,
            $$
            \ecat_{\alpha}(X)\le \etc(X)\le \ecat_{(\alpha,\alpha)}(X\times X)\le 2\,\ecat_\alpha(X).
            $$
            \item[(v)] For any exterior space $X$, $\nil\ker H^*_{\mathcal E}(\Delta_X)\le\etc(X)$.
                 \item[(vi)] For any exterior spaces $X$ and $Y$,
                $$
                \etc(X\times Y)\le \etc(X)+\etc(Y).
                $$
\end{itemize}
\end{theorem}

\begin{proof}  {\em (i)} This is immediate in view of (\ref{etcdiago}), Proposition \ref{propo2} and the exterior commutative square
$$\xymatrix{
X \ar[r]^{\simeq_{e} } \ar[d]_{\Delta_X} &
Y \ar[d]^{\Delta_Y}     \\
X\times X \ar[r]^{\simeq_{e} }&
Y\times Y }$$
produced by any given exterior homotopy equivalence $X\stackrel{\simeq_{e}}{\longrightarrow}Y$.

{\em (ii)} This is also trivial, again  in view of (\ref{etcdiago}) and Proposition \ref{propo1}, taking into account that $\tc(X)=\secat(\Delta_X)$.

{\em (iii)} If $X\simeq_e\br_+$ then $X\times X\simeq_e\br_+\times\br_+\simeq_e\br_+$ and thus, $\etc(X)=\esecat(\Delta)=\esecat(\id_{\br_+})=0$. Conversely, assume that $\etc(X)=0$ for  $X\in \exte_0$. This means that $q$ has an exterior section $s\colon X\times X\to X^I$. Let $\alpha\colon \br_+\to X$ be a ray of $X$ and $r\colon X\to\br_+$ be an exterior map. Then, the exterior map
$$
\Phi\colon X\longrightarrow X^I,\quad \Phi=s\circ(\id_X,\alpha\circ r),
$$
 is an exterior homotopy between $\id_X$ and $s\circ r$. On the other hand, one always have $r\circ s\simeq_e\id_{\br_+}$.

{\em (iv)}  In view of (\ref{etcdiago}), the second inequality of {\em(iv)} follows from Corollary \ref{coro1} while the third is a particular instance of Corollary \ref{catepro}.

For the first inequality, recall from \S1 that  in the diagram of exterior spaces and maps
$$\xymatrix{
 & {\mathbb{R}_+} \ar[dl]_{(\alpha ,\id_{\mathbb{R}_+})} \ar[dr]^{\alpha
 } & \\
 {X\times \mathbb{R}_+} \ar[rr]_{p}^{\simeq_e} & & {X}  }$$ the projection $p$ is an exterior homotopy equivalence.
Thus, $\ecat_{\alpha}(X)=\ecat_{(\alpha ,\id_{\mathbb{R}_+})}(X\times
\mathbb{R}_+)$. We finish by checking that
 $\ecat_{(\alpha
,\id_{\mathbb{R}_+})}(X\times \mathbb{R}_+)\leq \etc(X)$. For it, let
$U\subseteq X\times X$ be an open subset together with a strictly
commutative diagram in $\mathbf{E}$
$$\xymatrix{
{U} \ar@{^{(}->}@<-2pt>[rr] \ar[dr]_{\sigma } & & {X\times X} \\
 & {X^I} \ar[ur]_{q } }$$
and consider the open subset $V=\{(x,t)\in X\times \mathbb{R}_+,\,\,\bigl(x,\alpha (t)\bigr)\in U\}$. Then, the exterior map   $H\colon V\bar{\times
}I\rightarrow X\times \mathbb{R}_+$, $H\bigl((x,t),s\bigr)=\bigl(\sigma
\bigl(x,\alpha (t)\bigr)(s),t\bigr)$, is an exterior homotopy between the inclusion $V\hookrightarrow X\times \mathbb{R}_+$ and  the
composition $(\alpha ,\id_{\mathbb{R}_+})\circ \rho$ where $\rho$ is just the restriction to $V$ of the projection $X\times\br_+\to\br_+$. Applying this argument to an e-sectional cover of $X\times X$ finishes the proof.

{\em (v)} This is trivial in view of (\ref{etcdiago}) and Proposition \ref{cohomo}.

{\em (vi)} In view of (\ref{etcdiago}) this is just a particular case of Proposition \ref{producto}.
\end{proof}

The translation of the previous result to the proper setting provides:

\begin{theorem}\label{corotcp}
\begin{itemize}
\item[(i)] $\ptc$ is an invariant of the proper homotopy type.
\item[(ii)] $\tc(X)\le \ptc(X)$ for any  space $X$  and equality holds if $X$ is compact.
    \item[(iii)] If $X$ has the proper homotopy type of $\br_+$, then $\ptc(X)=0$. The converse also holds for any for any $X\in\pro_\infty$.
        \item[(iv)] For any  space $X\in\pro_\infty$ and any  ray $\alpha$ of $X$,
            $$
            \pcat_{\alpha}(X)\le \ptc(X)\le 2\,\pcat_\alpha(X).
            $$
            \item[(v)]  For any space $X$, and the exterior diagonal $\Delta_X\colon X_{cc}\to X_{cc}\times X_{cc}$, $\nil\ker H^*_{\mathcal E}(\Delta_X)\le \ptc(X)$.
\end{itemize}
\end{theorem}

\begin{proof}
Items {\em (i)}, {\em (ii)}, {\em (iii)} and the first inequality of {\em (iv)}  are immediately deduced by restricting the corresponding items in Theorem \ref{propie} to the cocompact externology and applying Proposition \ref{princi}.

For the second inequality of {\em (iv)} recall from (\ref{ptcdiago}) that $\ptc(X)=\esecat(\Delta_{X_{cc}})$ which is smaller than or equal to $\ecat_{\Delta\alpha}(X_{cc}\times X_{cc})$ in view of Corollary \ref{coro1}. But $\ecat_{\Delta\circ\alpha}(X_{cc}\times X_{cc})= \ecat_{(\alpha,\alpha)}(X_{cc}\times X_{cc})$ which, by Corollary \ref{catepro}, is bounded by $2\ecat_{\alpha}(X_{cc})=2\pcat_\alpha(X)$.

Finally,  {\em (v)} also holds by the corresponding item in Theorem \ref{propie} taking into account formula (\ref{ptcdiago}).
\end{proof}

\begin{remark} Note that  the second inequality   in {\em (iv)} of Theorem \ref{propie} does not translate to the proper setting as $\ecat_{(\alpha,\alpha)}(X_{cc}\times X_{cc})$ is different in general from $\pcat_{(\alpha,\alpha)}(X\times X)$, see Remark \ref{remi}.  A similar obstruction prevents recovering  {\em (vi)} of Theorem \ref{propie} in the proper context.
\end{remark}

\begin{theorem}\label{ultimo}
Let  $X$ be a pointed compact space. Then:
\begin{itemize}
\item[(i)] $
\ptc(X\times\br_+)=\ptc(X\vee\br_+)=\tc(X)$.
\item[(ii)] $\ptc(X\times\br)=\ptc(X\vee\br)=\infty$.
\end{itemize}
\end{theorem}
\begin{proof}
{\em (i)} Clearly, $\tc(X)=\tc(X\times\br_+)\le \ptc(X\times\br_+)$.
For the other inequality we see first that
\begin{equation}\label{igual2}
\ptc(X\times\br_+)=\psecat(\Delta_X\times\id).
\end{equation}
To this end notice that the map
$$
\varphi\colon (X\times X\times\br_+)_{cc}\longrightarrow (X\times\br_+)_{cc}\times(X\times\br_+)_{cc}, \quad\varphi(x,y,\lambda)=(x,\lambda,y,\lambda),
$$
is an exterior homotopy equivalence having
$$
\psi\colon (X\times\br_+)_{cc}\times(X\times\br_+)_{cc}\longrightarrow (X\times X\times\br_+)_{cc}, \quad\psi(x,\lambda,y,\mu)=(x,y,\lambda),
$$
as homotopy inverse. Indeed, both are exterior maps and
 $\psi\circ\varphi=\id$. On the other hand, the map
 $$
 H\colon\bigl((X\times\br_+)_{cc}\times(X\times\br_+)_{cc}\bigr)\times I_{tr}\to (X\times\br_+)_{cc}\times(X\times\br_+)_{cc},
 $$
 defined as $ H(x,\lambda,y,\mu,t)=(x,\lambda,y,(1-t)\mu+t\lambda)$
 is an exterior homotopy between $\varphi\circ\psi$ and $\id$.

 Therefore, in view of the commutative diagram of exterior spaces

$$
\xymatrix{
&(X\times\br_+)_{cc}\ar[dl]_{\Delta_X\times\id}\ar[dr]^{\Delta_{X\times\br_+}}&\\
(X\times X\times\br_+)_{cc}\ar[rr]^{\varphi}_{\simeq_e}&&(X\times\br_+)_{cc}\times(X\times\br_+)_{cc}
}
$$
we deduce that
 $$
 \begin{aligned}
 \ptc(X\times\br_+)&=\etc(X\times\br_+)_{cc}=\esecat(\Delta_{X\times\br_+})\\&=
 \esecat(\Delta_{X}\times\id)=\psecat(\Delta_X\times\id)
 \end{aligned}
 $$

Assume now $\tc(X)=n$ and choose $\{V_i\}_{i=0}^n$ an open cover of $X\times X$ together with homotopy commutative diagrams
$$\xymatrix{
\overline V_i \ar@{^{(}->}@<-2pt>[rr]\ar[dr]_{s_i} & & {X\times X} \\
 & {X} \ar[ur]_{\Delta_X }}$$

Recall that for any continuous map $f\colon Y\to Z$ between compact spaces, the map $f\times\id_{\br_+}\colon Y\times\br_+\to Z\times\br_+$ is proper and this defines  a functor from the category of compact spaces to $\PP$ which preserves homotopies. Applying this functor to the above diagrams we get new ones in $\PP$ which are commutative up to proper homotopy
$$\xymatrix{
\overline V_i\times\br_+ \ar@{^{(}->}@<-2pt>[rr]\ar[dr]_{s_i\times\id} & & {X\times X\times\br_+} \\
 & {X\times\br_+} \ar[ur]_{\Delta_X \times\id }}$$
 Since $\overline V_i\times\br_+=\overline{V_i\times\br}_+$ we deduce that $\psecat(\Delta_X\times\id)\le n$, or equivalently in view of (\ref{igual2}), $\ptc(X\times\br_+)\le n$.

 An analogous argument proves the equality $\ptc(X\vee\br_+)=\tc(X)$, where $0$ is the chosen  base point in $\br_+$.

 {\em (ii)} We show first that $\ptc(X\times\br)=\infty$. By Proposition \ref{princi} this is equivalent to prove that $\etc\bigl((X\times\br)_{cc}\bigr)=\infty$. We first note the following: let $Y$ and $Z$ be exterior spaces with rays $\alpha$ and $\beta$ respectively such that $(Y,\alpha)$ is an exterior homotopy retract of $(Z,\beta)$, that is, there exists an exterior homotopy commutative diagram of the form
$$\xymatrix{
 & {\mathbb{R}_+} \ar[d]^{\beta } \ar[dl]_{\alpha } \ar[dr]^{\alpha} \\
 {Y} \ar[r]_{\lambda } & {Z} \ar[r]_{\mu } & {Y}
}$$ \noindent where $\mu \circ \lambda \simeq _e \id_{Y}$. Then,
\begin{equation}\label{ufff}
\ecat_{\alpha}(Y)\le \ecat_{\beta}(Z).
\end{equation}
Indeed, assume that there is a diagram like (\ref{uff}) with $r\colon U\to\br_+$ such that $\beta\circ r$ is exterior homotopic to the inclusion $U\subseteq Z$. Then, one easily checks that, for the exterior map $r\circ\lambda\colon \lambda^{-1}(U)\to \br_+$, the composition $\alpha \circ r\circ \lambda$ is homotopic to the inclusion   $\lambda^{-1}(U)\subseteq Y$, which proves (\ref{ufff}).

Consider the particular case in which: $Y=\br_{cc}$; $\alpha$ is the canonical inclusion; $Z=(X\times\br)_{cc}$; $\beta(t)=(x_0,t)$ for any fixed $x_0\in X$; $\lambda(t)=(x_0,t)$ and $\mu$ is the projection. Then, applying (\ref{ufff}) and  the first inequality of {\em (iv)} in Theorem \ref{propie},
$$
\etc\bigl((X\times\br)_{cc}\bigr)\ge \ecat_\beta\bigl((X\times\br)_{cc}\bigr)\ge\ecat_\alpha(\br_{cc}).
$$
However, see (\ref{catego}), $\ecat_\alpha(\br_{cc})=\pcat_{\alpha}(\br)$ which is known to be infinite.

For the equality $\ptc(X\vee\br)=\infty$ we may assume without loosing generality that $0$ is the  base point of $\br$. Then, follow the previous argument choosing $Z=X\vee \br$.
\end{proof}

\begin{theorem}\label{noncompact}
For any non-empty compact space $X$, and any $n\ge1$,
$$
 \ptc(X\times\br^n)\ge \tc(X\times S^{n-1}).
$$
\end{theorem}
\begin{proof} For $n=1$ this is trivial in view of Theorem \ref{ultimo}{\em(ii)}. Fix $n\ge2$,
assume $\ptc(X\times\br^n)=m$ and let $\{U_i\}_{i=0}^m$ be an open covering of $(X\times\br^n)^2$ for which there exist local sections $\alpha_i\colon U_i\to X\times\br^n$, $i=0,\dots,m$,
of the  path fibration
$
q\colon ({X\times \br^n})^I\to (X\times \br^n)^2,
$
 satisfying the extra condition of Definition \ref{defmain2}. This requirement implies the existence of compact subspaces $\{K_i\}_{i=0}^m$ of $\br^n$ containing the compact $X\times\{0\}$ of $X\times\br^n$ such that, for all $i$,
$$
\im\alpha_i(x,y)\subseteq X\times (\br^n{\setminus}\{0\}),\quad\text{for all $x,y\in K_i^c$ with $(x,y)\in U_i$.}
$$
Choose $B$ a closed ball centered in $0\in \br^n$ so that  $K_i\subseteq X\times B$ for all $i$ and thus
$$
\im\alpha_i(x,y)\subseteq X\times (\br^n{\setminus}\{0\}),\quad\text{for all $x,y\in X\times B^c$ with $(x,y)\in U_i$.}
$$
That is, for all $i$,
$$
\alpha_i\colon U_i\cap\bigl((X\times B^c)\times (X\times B^c)\bigr)\to \bigl(X\times(\br^n{\setminus}\{0\})\bigr)^I
$$
Write $U'_i=U_i\cap\bigl((X\times B^c)\times (X\times B^c)\bigr)$ so that $\{U'_i\}_{i=0}^m$ is  an open covering of $X\times(\br^n{\setminus} B)$.

Consider
the homotopy retract
$
j\colon X\times (\br^n{\setminus} B)\stackrel{\simeq}{\hookrightarrow}  X \times(\br^n{\setminus} \{0\})
$
 which induces a homotopy equivalence
$$
j^I\colon  \bigl(X\times(\br^n{\setminus} B)\bigr)^I\stackrel{\simeq}{\longrightarrow} \bigl(X\times(\br^n{\setminus}\{0\})\bigr)^I
$$
fitting in the following commutative diagram
$$\xymatrix{
\bigl(X\times(\br^n{\setminus} B)\bigr)^I \ar[rr]^{j^I}_{\simeq} \ar[d]_q  & & {\bigl(X\times(\br^n{\setminus} 0)\bigr)^I} \ar[d]^{q} \\
{\bigl(X\times(\br^n{\setminus} B)\bigr)^2} \ar[rr]_{j\times j }^{\simeq} & & {\bigl(X\times(\br^n{\setminus} 0)\bigr)^2.} }$$
Denoting by $f$ the homotopy inverse of $j^I$  one easily checks that for each $i$ the map
$
f\circ\alpha_i\colon U'_i\to \bigl(X\times(\br^n{\setminus} B)\bigr)^I
$
is a homotopy local section of the path fibration in $\bigl(X\times(\br^n{\setminus} B)\bigr)^I$. Hence,
since $X\times(\br^n{\setminus} B)\simeq X\times S^{n-1}$, it follows that $\tc(X\times S^{n-1})\le m$.
\end{proof}

\section{Examples}

We now present a set of examples in which the proper topological complexity of some non compact spaces is explicitly computed.

\begin{proposition}\label{euclideo} Let $n\ge 1$ be and odd integer. Then,

$$
\ptc(\br^n)=\begin{cases}\infty,&\text{if $n=1$},\\
2,&\text{otherwise}.\end{cases}
$$
\end{proposition}

\begin{proof}
As recalled in \S1, and for any ray $\alpha$, $\pcat_\alpha(\br)=\infty$ and $\pcat_\alpha(\br^n)=1$ for any integer $n\ge2$. Hence, by {\em (iv)} of Theorem \ref{corotcp},
$
\ptc(\br)=\infty$ and
\begin{equation}\label{cotamain}
1\le\ptc(\br^n)\le2\quad\text{for $n\ge2$}.
\end{equation}
Given any odd integer $n>1$, choosing $X=\{*\}$ in Theorem \ref{noncompact} implies
$$
\ptc(\br^n)\ge\tc(S^{n-1})=2,
$$
which together with (\ref{cotamain}) provides equality.
\end{proof}

 \begin{remark}\label{remark2}
Odd dimensional Euclidean spaces provide examples for\break which $\psecat(\Delta_X)$ differs from $\ptc(X)$. Indeed,  by Corollary \ref{coro2}(iii),
$$
\psecat(\Delta_{\br^n})\le \pcat_{\Delta_{\br^n}\circ\alpha}( \br^{2n})=1
$$
for any $n\ge1$ any ray $\alpha$ of $\br^n$. But, using for instance compact supported cohomology, one easily sees that the proper map $\Delta_{\br^n}\colon \br^n\to \br^n\times \br^n$ does not admit a proper homotopy section. Thus,
 $$
 \psecat(\Delta_{\br^n})=1,
 $$
 while, $\ptc(\br^n)=2$ for any odd integer greater than 1.
\end{remark}

\begin{proposition} For any integers $n\ge 1$ and $m\ge 2$,
$$
2\le \ptc(S^n\times\br^m)\le 4.
$$
Moreover, if $n$ is even and $m\ge3$ is odd, the upper bound is attained,
$$
\ptc(S^n\times\br^m)= 4.
$$
\end{proposition}
\begin{proof} By Theorems 25 and 28 of \cite{gargarmu2},   any Euclidean space of dimension at least 2 satisfies the {\em proper Ganea  conjecture}  so that, for any ray $\alpha$ of $\br^m$,
$$
\pcat_{(*,\alpha)}(S^n\times\br^m)=\pcat_\alpha(\br^m)+1=2.
$$
With this, the first assertion follows directly from  {\em(iv)} of Theorem \ref{corotcp}.

On the other hand, the second assertion is an immediate from Theorem  \ref{noncompact}:
$$
\ptc(S^n\times\br^m)\ge \tc(S^n\times S^{m-1})=4.
$$
\end{proof}
\begin{remark} The previous result highly contrast with the equality
$$\ptc(S^n\times \br)=\infty
$$
immediately deduced from Theorem \ref{ultimo}{\em(ii)}.
\end{remark}

Finally, recall that, for any $n\ge 1$, the {\em Brown sphere} which gives rise to the {\em Brown-Grossman homotopy groups} of a given exterior space is defined as
$$
S_B^n=\br_+\cup\bigl( \cup_{k=0}^\infty\, S^n\times\{k\}\bigr)/\bigl(k\sim (x_0,k)\bigr)
$$
endowed with the cocompact externology, in which $x_0$ is the chosen base point in the $n$-sphere.
\begin{figure}[H]
\centering
\includegraphics[height=18mm]{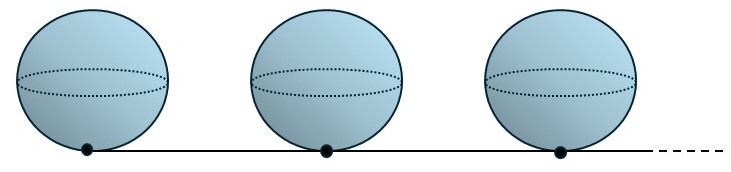}
\end{figure}

\begin{proposition} $\ptc(S^n_B)=2$ for any $n\ge 1$.
\end{proposition}
\begin{proof}
Note that $S^n_B$ has the homotopy type of an infinite wedge of $n$-spheres which can be retracted to $S^n\vee S^n$. Thus, by {\em (ii)} of Theorem \ref{corotcp}, Lemma 4.25 of \cite{far2} and Theorem 3.6 of \cite{dra2},
$$
\begin{aligned}
\ptc(S_B^n)&\ge\tc(\vee_{k=1}^\infty S^n)\ge\tc(S^n\vee S^n)\\&\ge \max\{\tc(S^n),\cat{S^n\times S^n}\}=2.
\end{aligned}
$$
On the other hand, for any ray $\alpha$, $\pcat_\alpha(S^n_B)=1$ so {\em (iv)} of Theorem \ref{corotcp} implies that $\ptc(S^n_B)\le 2$.
\end{proof}

A similar argument provides:

\begin{proposition}\label{sacabo} Let $S$ be a one-ended non-compact connected surface not homeomorphic to $\br^2$. Then:
$$
2\le \ptc(S)\le 4.
$$
\end{proposition}

\begin{proof}
In general, see for instance \cite{nara}, any non-compact $n$-manifold is of the homotopy type of an $(n-1)$-complex. In particular a non-compact, non-contractible surface $S$ has the homotopy type of a wedge of at least two circles. Hence, the same argument as in the preceding result yields $2\le \ptc (S)$.

On the other hand, by \cite[Corollary 3.3]{adomarquin} together with \cite[Remark 3.3]{carlasmuquin}, there exists a ray $\alpha$ in $S$ for which $\pcat_\alpha(S)=2$. Hence $\ptc(S)\le 4$.
\end{proof}
\begin{remark}
{\em (i)} Here, non-compact surfaces are considered in their most general sense. See \cite[Theorem 3]{ri} for their classification. The complement of a Cantor set in a closed surface is an example.

{\em (ii)} Note also that the previous result contrast with the equality
$$\ptc(\br^2{\setminus}\{*\})=\infty
$$
deduced from Theorem \ref{ultimo}{\em(ii)} in view of the homeomorphism $\br^2{\setminus}\{*\}\cong S^1\times\br$. Note that this open manifold has two ends so it does not fit in the hypotheses of Proposition \ref{sacabo}
\end{remark}

\section{A final application}\label{finale}

We conclude with a framework which can be directly applied to robotics problems: Let $Y$ represent the configuration space of a given motion planning problem, subject to the condition that any path representing the motion between two points in a specific subspace $Z\subseteq Y$ must remain entirely within $Z$.

This gives rise to the following notion, not to be confused with the relative versions of the topological complexity in \cite[Definition 4.20]{far2} or \cite[Definition 2.5]{short}:

\begin{definition}\label{rtc} Let $(Y, Z)$ be a pair of spaces. The \emph{relative topological complexity} of $(Y, Z)$, denoted by $\tc(Y, Z)$, is the smallest integer $n$ (or $\infty$ if no such integer exists) such that $Y \times Y$ can be covered by $n+1$ open sets $\{U_0, \dots, U_n\}$, with continuous local sections $\sigma_i \colon U_i \to Y^I$ of the path fibration satisfying $\sigma_i(z_1, z_2) \in Z^I$ whenever $z_1, z_2 \in Z$.
\end{definition}

Note that $\tc(Y) = \tc(Y, \emptyset)$.

In many cases, this notion is closely related to the proper topological complexity. Namely:

\begin{proposition} Let $Y$ be a finite cell complex, and let $Z\subseteq Y$ be a subcomplex such that the inclusion $i\colon Y{\setminus} Z\stackrel{\simeq}{\hookrightarrow} Y$ is a homotopy equivalence. Then,
 $$
 \ptc(Y{\setminus}Z)=\tc(Y,Z).
 $$
 \end{proposition}
Observe that the above result applies, for instance, to the case where Y is a compact manifold with boundary and $Z$ is a subcomplex of the boundary.

 \begin{proof}
We may assume that $Z$ has a small neighborhood $W$ in $Y$ such that the inclusion $W \setminus Z \stackrel{\simeq}{\hookrightarrow} W$ is a homotopy equivalence. To construct this, begin by selecting a small neighborhood $W'$ of $Z$ in $Y$. Then, define $Y' = Y$, set $Z' = \partial W'=\overline{W'}\setminus W' $, and equip them with finite cell complex structures such that $Z'$ is a subcomplex of $Y'$. Observe that $Z'$ now has a small neighborhood satisfying the desired property, and $Y' \setminus Z'$ is homeomorphic to $Y \setminus Z$. Note that this property is already satisfied if $Y$ is a compact manifold and $Z$ is a subcomplex of the boundary.

In particular, the existence of $W$ as described above allows us to choose a homotopy inverse $r\colon Y\stackrel{\simeq}{\to} Y\sm Z$ and homotopies $F\colon \id_{Y\sm Z}\simeq r\circ i$ and  $G\colon \id_Y\simeq i\circ r$, which preserve the proximity of points near $Z$.

Next, let  $s\colon U\to  Y^I$  be a section of the path fibration such that the image of the path $s(z_1,z_2)$ is contained in $Z$ whenever $z_1,z_2\in Z$. We define $V=U\cap(Y\setminus Z\times Y\setminus Z)=(i\times i)^{-1}(U)$ and set the section $\sigma\colon V\to (Y\sm Z)^I$ of the path fibration as follows: for any $x,y\in V$,
$$
 \sigma(x,y)=F(x,-)\cdot \bigl(r\circ s(x,y)\bigr)\cdot F(y,-)^{-1}.
 $$
In other words, $\sigma(x,y)$ is the product of three paths that join consecutively $x$ with $r(x)$, $r(x)$ with $r(y)$ and $r(y)$ with $y$. Then, in view of the above choices of $r$ and $F$, and considering that any point of  $\im s(x,y)$ is near $Z$ whenever $x$ and $y$ are, any compact $K$ in $Y\setminus Z$ is contained in another compact $L$ for which $\im\sigma(x,y)$ does not intersect $K$ whenever $x,y\notin L$. This implies, in particular, that $\ptc(Y\setminus Z)\le \tc(Y,Z)$.

Conversely, let $\sigma\colon V\to (Y\sm Z)^I$ be a section of the path fibration such that any compact $K$ of $Y\sm Z$ is contained in another compact $L$ such that $\im\sigma(x,y)$ does not intersect $K$ whenever $x,y\notin L$. Set $U=(r\times r)^{-1}(V)$ and define the section $s\colon U\to Y^I$ of the path fibration by
$$
 s(x,y)= G(x,-)\cdot \sigma\bigl(r(x),r(y)\bigr)\cdot G(y,-)^{-1}.
$$
By the hypothesis on $\sigma$ and the properties of $r$ and $G$, any point of $\im s(x,y)$ is near $Z$ whenever $x$ and $y$ are. In particular, $\im s(z_1,z_2)\subseteq Z$ if $z_1,z_2\in Z$.
 \end{proof}

 \begin{remark}  Observe that in the proposition above we may consider $Z$ and $Y\setminus Z$ as the unique safe and unsafe regions, respectively, in the motion planning. Furthermore, this result indicates that, in all cases it encompasses, which are likely to be the relevant ones in applications, the choice of the compact unsafe region is not essential.

Based on the above, the concept of relative topological complexity merits further development and study, as it may prove to be the appropriate framework for implementing precise constraints in a given motion planning problem.

 \end{remark}

\bigskip\bigskip\bigskip

\noindent {\sc Departamento de Matem\'aticas, Estad{\'\i}stica e Investigaci\'on Ope-\break\noindent rativa, Universidad de La Laguna, Avenida Astrof{\'\i}sico Francisco S\'anchez S/N,
38200 La Laguna, Spain}.

\noindent\texttt{jmgarcal@ull.edu.es}

\bigskip

\noindent{\sc Departamento de \'Algebra, Geometr\'{\i}a y Topolog\'{\i}a, Universidad de M\'alaga, Bulevar Louis Pasteur 31, 29010 M\'alaga, Spain.}

\noindent
\texttt{aniceto@uma.es}

\end{document}